\documentclass[11pt,a4paper]{article}
\usepackage{amsmath}
\usepackage{amsfonts}

\newtheorem{theorem}{Theorem}

\newtheorem{corollary}[theorem]{Corollary}

\newtheorem{lemma}[theorem]{Lemma}

\newenvironment{proof}[1][Proof]{\noindent\textbf{#1.} }{\ \rule{0.5em}{0.5em}}
\input{tcilatex}

\begin{document}

\title{A strong form of Plessner's theorem}
\date{}
\author{Stephen J. Gardiner and Myrto Manolaki}
\maketitle

\begin{abstract}
Let $f$ be a holomorphic, or even meromorphic, function on the unit disc.
Plessner's theorem then says that, for almost every boundary point $\zeta $,
either (i) $f$ has a finite nontangential limit at $\zeta $, or (ii) the
image $f(S)$ of any Stolz angle $S$ at $\zeta $ is dense in the complex
plane. This paper shows that statement (ii) can be replaced by a much
stronger assertion. This new theorem and its analogue for harmonic functions
on halfspaces also strengthen classical results of Spencer, Stein and
Carleson.
\end{abstract}

\section{Introduction}

\footnotetext{%
\noindent \noindent 2010 \textit{Mathematics Subject Classification } 30D40,
30K05, 30K10 31B25.
\par
\noindent \textit{Keywords: }holomorphic functions, meromorphic functions,
boundary behaviour, Plessner's theorem, harmonic functions.}Let $\mathbb{D}$
be the unit disc of the complex plane, $\mathbb{T}$ be the unit circle, and $%
\lambda _{n}$ denote Lebesgue measure on $\mathbb{R}^{n}$ $(n\geq 1)$, where
we identify $\mathbb{C}$ with $\mathbb{R}^{2}$, and $\mathbb{T}$ with $%
[0,2\pi )$, in the usual way. By a Stolz angle at a point $\zeta $ of $%
\mathbb{T}$ we mean an open triangular subset of $\mathbb{D}$ that has a
vertex at $\zeta $ and is symmetric about the diameter of $\mathbb{D}$\
through $\zeta $. A fundamental result concerning the boundary behaviour of
holomorphic functions is as follows (see the original paper \cite{Ple} or
any of the books \cite{CL}, \cite{Pomm}, \cite{GM}).

\bigskip

\noindent \textbf{Theorem A (Plessner) }\textit{Let }$f$\textit{\ be a
holomorphic function on }$\mathbb{D}$.\textit{\ Then, for }$\lambda _{1}$-%
\textit{almost every point }$\zeta $ \textit{of }$\mathbb{T}$,\textit{\
either}\newline
\textit{(i)\ }$f$\textit{\ has a finite nontangential limit at }$\zeta $%
\textit{, or\ \newline
(ii) }$\overline{f(S)}=\mathbb{C}$\textit{\ for every Stolz angle }$S$%
\textit{\ at }$\zeta $\textit{.}

\bigskip

Although this result describes a stark dichotomy in boundary behaviour, it
has been a long-standing open question whether condition (ii) can be
significantly strengthened. For example, Collingwood, one of the authors of
the standard text \cite{CL} on cluster sets, asked over 50 years ago whether
the statement that $\overline{f(S)}=\mathbb{C}$ can be replaced by the much
stronger assertion that $\lambda _{2}(\mathbb{C}\backslash f(S))=0$ (Problem
5.20 in \cite{H67} or \cite{HL}; see also Problem 5.57 in \cite{HL}). Below
we give a different substantial improvement of Plessner's theorem. We denote
the circle $\{z\in \mathbb{C}:\left\vert z-w\right\vert =r\}$ by $C_{w,r}$.

\begin{theorem}
\label{main}Let $f$\ be a holomorphic function on $\mathbb{D}$.\ Then, for $%
\lambda _{1}$-almost every point $\zeta $ of $\mathbb{T}$,\ either\newline
(i)\ $f$\ has a finite nontangential limit at $\zeta $, or\ \newline
(ii) for every Stolz angle $S$\ at $\zeta $,%
\begin{equation}
\dint_{S\cap f^{-1}(C_{w,r})}\left\vert f^{\prime }(z)\right\vert \left\vert
dz\right\vert =\infty  \label{length}
\end{equation}%
for $\lambda _{3}$-almost every $(w,r)\in \mathbb{C}\times (0,\infty )$.
\end{theorem}

The integral in (\ref{length}) measures the total arc length of the image of 
$S\cap f^{-1}(C_{w,r})$ under $f$, taking account of multiplicities. Thus,
although this image is contained in the circle $C_{w,r}$, condition (ii)
makes the striking assertion that its length, counting multiplicities, is
infinite for almost every choice of $(w,r)$.

Plessner's theorem holds more generally for meromorphic functions $f$ on $%
\mathbb{D}$, and the same is true of Theorem \ref{main}. We will explain at
the end of the proof of Theorem \ref{main} how the argument can be adapted
to cover meromorphic functions as well.

We next recall a further classical result concerning the boundary behaviour
of holomorphic functions (Theorem 5 in \cite{Spe}; see also Theorem X.1.3 in 
\cite{GM}, and p.\thinspace 364 of the survey article \cite{St82} for its
wider significance).

\bigskip

\noindent \textbf{Theorem B (Spencer) }\textit{Let }$f$\textit{\ be a
holomorphic function on }$\mathbb{D}$.\textit{\ Then, for }$\lambda _{1}$-%
\textit{almost every point }$\zeta $ \textit{of }$\mathbb{T}$,\textit{\
either}\newline
\textit{(i)\ }$f$\textit{\ has a finite nontangential limit at }$\zeta $%
\textit{, or\ \newline
(ii) for every Stolz angle }$S$\textit{\ at }$\zeta $\textit{,}%
\begin{equation*}
\dint_{S}\frac{\left\vert f^{\prime }\right\vert ^{2}}{\left( 1+\left\vert
f\right\vert ^{2}\right) ^{2}}d\lambda _{2}=\infty .
\end{equation*}

\bigskip

If $f$ and $S$ are as above, then the co-area formula (see Section 3.4.3 of 
\cite{EG}, or Section 1.2.4 of \cite{Maz}) tells us that%
\begin{equation*}
\int_{S\cap \{a<\left\vert f-w\right\vert <b\}}\left\vert f^{\prime
}\right\vert ^{2}d\lambda _{2}=\int_{(a,b)}\int_{S\cap \left\{ \left\vert
f-w\right\vert =t\right\} }\left\vert f^{\prime }(z)\right\vert \left\vert
dz\right\vert ~d\lambda _{1}(t)
\end{equation*}%
whenever $w\in \mathbb{C}$ and $0\leq a<b$. Hence Theorem \ref{main} is also
much stronger than Spencer's result. In particular, condition (ii) of
Theorem \ref{main} clearly implies that, for every Stolz angle $S$,%
\begin{equation*}
\int_{S\cap \{\left\vert f-w\right\vert <r\}}\left\vert f^{\prime
}\right\vert ^{2}d\lambda _{2}=\infty \text{ \ \ \ }(w\in \mathbb{C},r>0).
\end{equation*}%
(This last integral measures the total area of the image of $S\cap
\{\left\vert f-w\right\vert <r\}$ under $f$, counting multiplicities.)

Theorem \ref{main}, and its generalization to meromorphic functions, will be
proved in the next section, after which we will briefly discuss its
application to the theory of universal series. The final section of the
paper presents an analogue of Theorem \ref{main} for harmonic functions on
halfspaces, which strengthens well known results of Carleson and Stein.

\section{Proof of Theorem \protect\ref{main}}

We define nontangential approach regions at points $\zeta $ of $\mathbb{T}$
by 
\begin{equation*}
S(\zeta ,\delta )=\left\{ z\in \mathbb{D}:\sqrt{1-\delta ^{2}}\left\vert
z-\zeta \right\vert <1-\left\vert z\right\vert <\delta \right\} \text{ \ \ \ 
}(0<\delta <1).
\end{equation*}

Let $f$ be a holomorphic function on $\mathbb{D}$. Then it is well known
that the set of points in $\mathbb{T}$ at which $f$ has a finite
nontangential limit is a Borel subset of $\mathbb{T}$. It follows easily
from the lemma below that the same can also be said of the set of points $%
\zeta $ in $\mathbb{T}$ for which condition (ii) of Theorem \ref{main}
holds. Let 
\begin{equation}
L_{j}(\zeta ,w,r)=\dint_{S(\zeta ,j^{-1})\cap f^{-1}(C_{w,r})}\left\vert
f^{\prime }(z)\right\vert \left\vert dz\right\vert \text{ \ \ \ }(\zeta \in 
\mathbb{T},w\in \mathbb{C},r>0)  \label{Lj}
\end{equation}%
for each $j\in \{2,3,...\}$.

\begin{lemma}
\label{Borel}The function $L_{j}$ is Borel measurable on $\mathbb{T}\times 
\mathbb{C\times }(0,\infty )$ for each $j\in \{2,3,...\}$.
\end{lemma}

\begin{proof}
We dismiss the trivial case where $f$ is constant and so $L_{j}\equiv 0$.
Let $Z=\{z\in \mathbb{D}:f^{\prime }(z)=0\}$ and%
\begin{equation*}
A_{m}=\{z\in \mathbb{D}:\mathrm{dist}(z,Z)>m^{-1}\text{ \ and \ }\left\vert
z\right\vert <1-m^{-1}\}\text{ \ \ \ }(m=2,3,...),
\end{equation*}%
and then let%
\begin{equation*}
L_{j}^{(m)}(\zeta ,w,r)=\dint_{S(\zeta ,j^{-1})\cap A_{m}\cap
f^{-1}(C_{w,r})}\left\vert f^{\prime }(z)\right\vert \left\vert
dz\right\vert \text{ \ \ \ }(\zeta \in \mathbb{T},w\in \mathbb{C},r>0).
\end{equation*}%
Since $L_{j}(\zeta ,w,r)=\lim_{m\rightarrow \infty }L_{j}^{(m)}(\zeta ,w,r)$%
, it will be enough to show that $L_{j}^{(m)}$ is Borel measurable on $%
\mathbb{T}\times \mathbb{C\times }(0,\infty )$ for any $m$.

We now fix both $j$ and $m$. For any open set $U$ such that $\overline{U}%
\subset \mathbb{D}\backslash Z$, we define%
\begin{equation*}
T_{U}(\zeta ,w,r)=\dint_{S(\zeta ,j^{-1})\cap A_{m}\cap U\cap
f^{-1}(C_{w,r})}\left\vert f^{\prime }(z)\right\vert \left\vert
dz\right\vert .
\end{equation*}%
Let $W$ be an open set satisfying $\overline{W}\subset S(\zeta ,j^{-1})\cap
A_{m}\cap U$. If $(\zeta _{k},w_{k},r_{k})\rightarrow (\zeta ,w,r)$ in $%
\mathbb{T}\times \mathbb{C\times }(0,\infty )$, then there exists $k_{0}$
such that 
\begin{equation*}
\overline{W}\subset S(\zeta _{k},j^{-1})\cap A_{m}\cap U\text{ \ \ \ }(k\geq
k_{0}).
\end{equation*}%
Further, if $f|_{U}$ is injective, then $\dint_{W\cap
f^{-1}(C_{w,r})}\left\vert f^{\prime }(z)\right\vert \left\vert
dz\right\vert $ is the total arc length of $f(W)\cap C_{w,r}$, and so 
\begin{eqnarray*}
\underset{k\rightarrow \infty }{\lim \inf }~T_{U}(\zeta _{k},w_{k},r_{k})
&\geq &\underset{k\rightarrow \infty }{\lim \inf }\dint_{W\cap
f^{-1}(C_{w_{k},r_{k}})}\left\vert f^{\prime }(z)\right\vert \left\vert
dz\right\vert \\
&\geq &\dint_{W\cap f^{-1}(C_{w,r})}\left\vert f^{\prime }(z)\right\vert
\left\vert dz\right\vert .
\end{eqnarray*}%
On enlarging $W$ we see that%
\begin{equation*}
\underset{k\rightarrow \infty }{\lim \inf }~T_{U}(\zeta
_{k},w_{k},r_{k})\geq T_{U}(\zeta ,w,r),
\end{equation*}%
whence $T_{U}$ is lower semicontinuous on $\mathbb{T}\times \mathbb{C\times }%
(0,\infty )$.

Now let $U=U_{1}\cup U_{2}$, where $U_{1},U_{2}$ are open sets such that $%
\overline{U}_{i}\subset \mathbb{D}\backslash Z$ and $f|_{U_{i}}$ is
injective for each $i$. Then 
\begin{equation*}
T_{U}(\zeta ,w,r)=T_{U_{1}}(\zeta ,w,r)+T_{U_{2}}(\zeta ,w,r)-T_{U_{1}\cap
U_{2}}(\zeta ,w,r),
\end{equation*}%
so $T_{U}$ is Borel measurable. Similarly, $T_{U}$ is Borel measurable when $%
U$ is any finite union of such open sets $U_{i}$.

Since $f$ is locally injective on $\mathbb{D}\backslash Z$, we may, by
compactness, choose open sets $U_{1},...,U_{l}$ such that $\overline{U}%
_{i}\subset \mathbb{D}\backslash Z$ and $f|_{U_{i}}$ is injective for each $%
i $, and also $\overline{A}_{m}\subset U$, where $U=\cup _{i=1}^{l}U_{i}$.
Then $L_{j}^{(m)}(\zeta ,w,r)=T_{U}(\zeta ,w,r)$ and so, by the previous
paragraph, $L_{j}^{(m)}$ is Borel measurable on $\mathbb{T}\times \mathbb{%
C\times }(0,\infty )$, as required.
\end{proof}

\bigskip

\begin{proof}[Proof of Theorem \protect\ref{main}]
Let $f$ be a holomorphic function on $\mathbb{D}$. We may assume that $f$ is
nonconstant. The above lemma tells us that the function $L_{j}$ defined by (%
\ref{Lj}) is Borel measurable on $\mathbb{T}\times \mathbb{C\times }%
(0,\infty )$ for each $j$. Further, the sequence $(L_{j})$ is decreasing.
The set of all points $\zeta $ in $\mathbb{T}$ satisfying condition (ii) of
Theorem \ref{main} is given by $\cap _{j}E_{j}$, where 
\begin{equation*}
E_{j}=\left\{ \zeta \in \mathbb{T}:\int_{\mathbb{C}\times (0,\infty )}\chi
_{\{L_{j}<\infty \}}(\zeta ,w,r)~d\lambda _{3}(w,r)=0\right\} \text{ \ \ \ }%
(j\geq 2),
\end{equation*}%
and $(E_{j})$ is a decreasing sequence. The sets $E_{j}$ are Borel because
the functions $L_{j}$ are Borel. Let $F$ be the Borel set of all points in $%
\mathbb{T}$ at which $f$ has a finite nontangential limit.\ Theorem \ref%
{main} will follow if we can show that $\lambda _{1}(\mathbb{T}\backslash
(E_{j}\cup F))=0$ for each $j$. We now suppose, for the sake of
contradiction, that there exists $j$ such that $\lambda _{1}(B_{1})>0$,
where $B_{1}=\mathbb{T}\backslash (E_{j}\cup F)$. We proceed below with this
particular choice of $j$.

For each $\zeta \in B_{1}$ we know that 
\begin{equation*}
\int_{\mathbb{C}\times (0,\infty )}\chi _{\{L_{j}<\infty \}}(\zeta
,w,r)~d\lambda _{3}(w,r)>0.
\end{equation*}%
Hence we can choose $\rho >0$ large enough to ensure that $\lambda
_{1}(B_{2})>0$, where 
\begin{equation*}
B_{2}=\left\{ \zeta \in B_{1}:\int_{\mathbb{C}\times (0,\infty )}\chi
_{\{L_{j}\leq \rho \}}(\zeta ,w,r)~d\lambda _{3}(w,r)>0\right\} .
\end{equation*}%
Let%
\begin{equation*}
A(w,r)=\left\{ \zeta \in B_{2}:L_{j}(\zeta ,w,r)\leq \rho \right\} \text{ \
\ \ }(w\in \mathbb{C},r>0).
\end{equation*}%
Then%
\begin{eqnarray*}
\int_{\mathbb{C}\times (0,\infty )}\lambda _{1}(A(w,r))~d\lambda _{3}(w,r)
&=&\int_{\mathbb{C}\times (0,\infty )}\int_{B_{2}}\chi _{\{L_{j}\leq \rho
\}}(\zeta ,w,r)~d\lambda _{1}(\zeta )d\lambda _{3}(w,r) \\
&=&\int_{B_{2}}\int_{\mathbb{C}\times (0,\infty )}\chi _{\{L_{j}\leq \rho
\}}(\zeta ,w,r)~d\lambda _{3}(w,r)d\lambda _{1}(\zeta ) \\
&>&0,
\end{eqnarray*}%
by Tonelli's theorem and the choice of $B_{2}$. In particular, there exist $%
w_{0}\in \mathbb{C}$ and $r_{0}>0$ such that $\lambda _{1}(A(w_{0},r_{0}))>0$%
. For each $\zeta \in A(w_{0},r_{0})$ we know that 
\begin{equation}
\dint_{S(\zeta ,j^{-1})\cap f^{-1}(C_{w_{0},r_{0}})}\left\vert f^{\prime
}(z)\right\vert \left\vert dz\right\vert \leq \rho .  \label{A}
\end{equation}

We define arcs in $\mathbb{T}$ by%
\begin{equation*}
I(z,\delta )=\left\{ \eta \in \mathbb{T}:z\in S(\eta ,\delta )\right\} \text{
\ \ \ }(0<\delta <1,1-\delta <\left\vert z\right\vert <1).
\end{equation*}%
Thus $\zeta \in I(z,\delta )$ whenever $z\in S(\zeta ,\delta )$, and 
\begin{equation}
\frac{\lambda _{1}\left( I(z,\delta )\right) }{1-\left\vert z\right\vert }%
\rightarrow \frac{2\delta }{\sqrt{1-\delta ^{2}}}\text{ \ \ \ }(\left\vert
z\right\vert \rightarrow 1-).  \label{proj}
\end{equation}%
It follows from the Lebesgue density theorem that, for $\lambda _{1}$-almost
every point $\zeta $ of $A(w_{0},r_{0})$,%
\begin{equation*}
\underset{%
\begin{array}{c}
\left\vert z\right\vert \rightarrow 1- \\ 
z\in S(\zeta ,j^{-1})%
\end{array}%
}{\lim \inf }\frac{\lambda _{1}\left( I(z,j^{-1})\cap A(w_{0},r_{0})\right) 
}{\lambda _{1}\left( I(z,j^{-1})\right) }\geq \frac{1}{2}
\end{equation*}%
and so 
\begin{equation*}
\underset{%
\begin{array}{c}
\left\vert z\right\vert \rightarrow 1- \\ 
z\in S(\zeta ,j^{-1})%
\end{array}%
}{\lim \inf }\frac{\lambda _{1}\left( I(z,j^{-1})\cap A(w_{0},r_{0})\right) 
}{1-\left\vert z\right\vert }\geq \frac{j^{-1}}{\sqrt{1-j^{-2}}}=\frac{1}{%
\sqrt{j^{2}-1}}\text{,}
\end{equation*}%
by (\ref{proj}). Since $\lambda _{1}(A(w_{0},r_{0}))>0$, we can choose $%
\varepsilon \in (0,j^{-1})$ and a subset $A_{0}$ of $A(w_{0},r_{0})$ such
that $\lambda _{1}(A_{0})>0$ and 
\begin{equation}
\lambda _{1}\left( I(z,j^{-1})\cap A(w_{0},r_{0})\right) \geq
j^{-1}(1-\left\vert z\right\vert )\text{ \ \ \ }(1-\varepsilon <\left\vert
z\right\vert <1,z\in \Omega ),  \label{c}
\end{equation}%
where 
\begin{equation*}
\Omega =\bigcup\limits_{\zeta \in A_{0}}S(\zeta ,j^{-1}).
\end{equation*}

We obtain a measure on $\mathbb{D}$ by defining 
\begin{equation*}
\mu =\Delta \log ^{+}\frac{\left\vert f-w_{0}\right\vert }{r_{0}}
\end{equation*}%
in the sense of distributions. It has support in the set $I=\{\left\vert
f-w_{0}\right\vert =r_{0}\}$. Also, $\left\Vert \nabla \log \left\vert
f-w_{0}\right\vert \right\Vert =\left\vert f^{\prime }\right\vert
/\left\vert f-w_{0}\right\vert $ wherever $f\neq w_{0}$, and $\nabla \log
\left\vert f-w_{0}\right\vert $ is normal to $I$ on the set $\{z\in
I:f^{\prime }(z)\neq 0\}$. Given any test function $\phi \in C_{c}^{\infty }(%
\mathbb{D})$, we can choose $s\in (0,1)$ so that the support of $\phi $ is
contained in $\{z:\left\vert z\right\vert <s\}$, and apply Green's identity
on the open set 
\begin{equation*}
U=\{z:\left\vert z\right\vert <s\text{ \ and \ }\left\vert
f(z)-w_{0}\right\vert >r_{0}\}
\end{equation*}%
to see that%
\begin{eqnarray*}
\left( \Delta \log ^{+}\frac{\left\vert f-w_{0}\right\vert }{r_{0}}\right)
(\phi ) &=&\int_{U}\left( \log ^{+}\frac{\left\vert f-w_{0}\right\vert }{%
r_{0}}\right) \Delta \phi ~d\lambda _{2} \\
&=&\int_{\{\left\vert f-w_{0}\right\vert =r_{0}\}}\phi (z)\left\Vert \nabla
\log \left\vert f-w_{0}\right\vert (z)\right\Vert \left\vert dz\right\vert \\
&=&\frac{1}{r_{0}}\int_{\{\left\vert f-w_{0}\right\vert =r_{0}\}}\phi
(z)\left\vert f^{\prime }(z)\right\vert \left\vert dz\right\vert .
\end{eqnarray*}%
Hence%
\begin{equation}
\mu (J)=\frac{1}{r_{0}}\int_{J\cap f^{-1}(C_{w_{0},r_{0}})}\left\vert
f^{\prime }(z)\right\vert \left\vert dz\right\vert \text{ \ \ \ for any open
set }J\subset \mathbb{D}.  \label{mu}
\end{equation}%
Further, by (\ref{A}),%
\begin{equation*}
\int_{S(\zeta ,j^{-1})}d\mu =\frac{1}{r_{0}}\int_{S(\zeta ,j^{-1})\cap
f^{-1}(C_{w_{0},r_{0}})}\left\vert f^{\prime }(z)\right\vert \left\vert
dz\right\vert \leq \frac{\rho }{r_{0}}\text{ \ \ \ }(\zeta \in
A(w_{0},r_{0})).
\end{equation*}

We now see that%
\begin{eqnarray*}
\frac{\rho }{r_{0}}\lambda _{1}(A(w_{0},r_{0})) &\geq
&\int_{A(w_{0},r_{0})}\int_{S(\zeta ,j^{-1})}~d\mu (\xi )d\lambda _{1}(\zeta
) \\
&\geq &\int_{A(w_{0},r_{0})}\int_{\Omega }\chi _{S(\zeta ,j^{-1})}(\xi
)~d\mu (\xi )d\lambda _{1}(\zeta ) \\
&=&\int_{\Omega }\int_{A(w_{0},r_{0})}\chi _{S(\zeta ,j^{-1})}(\xi
)~d\lambda _{1}(\zeta )d\mu (\xi ) \\
&=&\int_{\Omega }\lambda _{1}(I(\xi ,j^{-1})\cap A(w_{0},r_{0}))~d\mu (\xi )
\\
&\geq &j^{-1}\int_{\Omega \cap \left\{ 1-\varepsilon <\left\vert \xi
\right\vert <1\right\} }(1-\left\vert \xi \right\vert )~d\mu (\xi ),
\end{eqnarray*}%
by (\ref{c}). In particular, this last integral is finite, so we can define
the Green potential%
\begin{equation*}
u(z)=\int_{\Omega }G_{\mathbb{D}}(z,\xi )~d\mu (\xi )\text{ \ \ \ }(z\in 
\mathbb{D}),
\end{equation*}%
where $G_{\mathbb{D}}(\cdot ,\cdot )$ denotes the Green function for the
unit disc (see Theorem 4.2.5(ii) of \cite{AG}).

We know (see Corollary 4.3.3 and Theorems 4.3.8(i) and 4.3.5 of \cite{AG})
that the function%
\begin{equation}
h(z)=\log ^{+}\frac{\left\vert f(z)-w_{0}\right\vert }{r_{0}}+\frac{u(z)}{%
2\pi }\text{ \ \ \ }(z\in \mathbb{D})  \label{Lap}
\end{equation}%
is harmonic on $\Omega $, and clearly $h>0$. We can now combine a standard
conformal mapping argument (see the proof of Theorem VI.2.2 in \cite{GM})
with Fatou's theorem for positive harmonic functions on $\mathbb{D}$
(Theorem 4.6.7 in \cite{AG}) to see that $h$ has a finite nontangential
limit at $\lambda _{1}$-almost every point of $A_{0}$. Since $h\geq \log
(\left\vert f-w_{0}\right\vert /r_{0})$, it follows that $f$ is
nontangentially bounded $\lambda _{1}$-almost everywhere on $A_{0}$, and so $%
f$ has finite nontangential limits $\lambda _{1}$-almost everywhere on $%
A_{0} $ by a theorem of Privalov (Theorem VI.2.2 of \cite{GM}). However, 
\begin{equation*}
A_{0}\subset A(w_{0},r_{0})\subset B_{2}\subset B_{1}=\mathbb{T}\backslash
(E_{j}\cup F)\subset \mathbb{T}\backslash F,
\end{equation*}%
so $f$ cannot have a finite nontangential limit at any point of $A_{0}$, and 
$\lambda _{1}(A_{0})>0$. The assumption that $\lambda _{1}(B_{1})>0$ has
thus led to a contradiction. We conclude that $\lambda _{1}(\mathbb{T}%
\backslash (E_{j}\cup F))=0$ for every $j$, and so the theorem is
established.
\end{proof}

\bigskip

\noindent \textbf{Remark }Let $f$ be a holomorphic function on $\mathbb{D}$,
and $Y$ be a countable subset of $\mathbb{C}\times (0,\infty )$. Then, for $%
\lambda _{1}$-almost every point $\zeta $ of $\mathbb{T}$, either (i) $f$
has a finite nontangential limit at $\zeta $, or (ii) for every Stolz angle $%
S$ at $\zeta $ and every pair $(w,r)$ in $Y$ equation (\ref{length}) holds.
To see this, it is enough to consider the case where $Y$ is a singleton $%
\{(w_{0},r_{0})\}$. We can then follow the outline of the proof of Theorem %
\ref{main}, provided that we now define 
\begin{equation*}
E_{j}=\left\{ \zeta \in \mathbb{T}:L_{j}(\zeta ,w_{0},r_{0})=\infty \right\}
\end{equation*}%
and choose $\rho >0$ large enough so that $\lambda _{1}(A(w_{0},r_{0}))>0$,
where%
\begin{equation*}
A(w_{0},r_{0})=\left\{ \zeta \in \mathbb{T}\backslash (E_{j}\cup
F):L_{j}(\zeta ,w_{0},r_{0})\leq \rho \right\} .
\end{equation*}%
This variant of Theorem \ref{main} also implies Plessner's theorem, since we
may choose $Y$ to be dense in $\mathbb{C}\times (0,\infty )$.

\bigskip

\begin{proof}[Extension to meromorphic functions]
Now suppose that $f$ is merely meromorphic on $\mathbb{D}$. We note that
Lemma \ref{Borel} extends easily to this case, and we follow the proof of
Theorem \ref{main} as far as the sentence containing (\ref{c}). Then 
\begin{equation*}
\Delta \log ^{+}\frac{\left\vert f-w_{0}\right\vert }{r_{0}}=\mu -2\pi \mu
_{1},
\end{equation*}%
where $\mu $ again satisfies (\ref{mu}) and $\mu _{1}$ is the measure which
counts the poles of $f$ according to multiplicity. The function $h$ in (\ref%
{Lap}) now satisfies $\Delta h=-2\pi \mu _{1}$ on $\Omega $, and so is
superharmonic there. The conformal mapping argument that we used previously
thus yields a positive superharmonic function $v$ on $\mathbb{D}$ such that $%
(-\Delta v)/(2\pi )$ is a sum of Dirac measures. By the Riesz decomposition
theorem (Theorem 4.4.1 of \cite{AG}) and Theorem 4.2.5(ii) of \cite{AG}, the
discrete measure $-\Delta v$ satisfies the Blaschke condition on $\mathbb{D}$
and $v$ has the form $v=h_{1}-\log \left\vert B\right\vert $, where $h_{1}$
is a positive harmonic function on $\mathbb{D}$ and $B$ is a Blaschke
product associated with $(-\Delta v)/(2\pi )$. Hence $v$ has a finite
nontangential limit $\lambda _{1}$-almost everywhere on $\mathbb{T}$, and so 
$h$ has a finite nontangential limit at $\lambda _{1}$-almost every point of 
$A_{0}$. The proof concludes as before.
\end{proof}

\section{Application to universal series}

A power series $\sum a_{n}z^{n}$ with radius of convergence $1$ is said to
belong to the collection $\mathcal{U}$\textit{\ }if, for every compact set $%
K\subset \mathbb{C}\backslash \mathbb{D}$ with connected complement and
every continuous function $g:K\rightarrow \mathbb{C}$ that is holomorphic on 
$K^{\circ }$, there is a subsequence $(m_{k})$ of the natural numbers such
that $\sum_{0}^{m_{k}}a_{n}z^{n}\rightarrow g(z)$ uniformly on $K$. Members
of $\mathcal{U}$ are called \textit{universal Taylor series}, and such
functions have been widely studied in recent years. Nestoridis \cite{Nes}
showed that this universal approximation property is a generic feature of
holomorphic functions on the unit disc; more precisely, he proved that $%
\mathcal{U}$ is a dense $G_{\delta }$ subset of the space of all holomorphic
functions on $\mathbb{D}$ endowed with the topology of local uniform
convergence.

In Theorem 2 of \cite{Gar} nontrivial potential theoretic arguments were
used to show that universal Taylor series cannot have nontangential limits
at a boundary set of positive measure, and so must satisfy condition (ii) of
Plessner's theorem at $\lambda _{1}$-almost every point $\zeta $ of $\mathbb{%
T}$. If we substitute Theorem \ref{main} of this paper for Plessner's
theorem in the proof given in \cite{Gar}, we immediately obtain the
following improvement. It implies, in particular, that a generic property of
holomorphic functions on $\mathbb{D}$ is that condition (ii) of Theorem \ref%
{main} holds for almost every $\zeta $ in $\mathbb{T}$.

\begin{corollary}
\label{UTS}If $f\in \mathcal{U}$, then for $\lambda _{1}$-almost every point 
$\zeta $ of $\mathbb{T}$ equation (\ref{length}) holds for every Stolz angle 
$S$\ at $\zeta $ and $\lambda _{3}$-almost every $(w,r)\in \mathbb{C}\times
(0,\infty )$.
\end{corollary}

The same improvement may be made to Corollary 2 of \cite{GaMa0}, which
generalizes Theorem 2 of \cite{Gar}, and to Corollary 5 of \cite{GaMa},
which establishes the corresponding boundary behaviour of universal
Dirichlet series (in this case we would use the obvious analogue of Theorem %
\ref{main} for holomorphic functions in a halfplane).

\section{Boundary behaviour of harmonic functions}

We will now present an analogue of Theorem \ref{main} for harmonic functions
on the halfspace $\mathbb{H}=\{(x_{1},...,x_{N})\in \mathbb{R}^{N}:x_{N}>0\}$%
, where $N\geq 2$. By a Stolz domain at $y\in \partial \mathbb{H}$ we mean a
truncated cone in $\mathbb{H}$ that meets $\partial \mathbb{H}$ precisely at
its vertex $y$ and with its axis normal to $\partial \mathbb{H}$. We will
consider $\lambda _{N-1}$ as a measure on $\partial \mathbb{H}$ by
identifying this set with $\mathbb{R}^{N-1}$ in the obvious way.

If $h$ is a harmonic function on $\mathbb{H}$, then a result of Carleson 
\cite{Car} is equivalent to saying that, for $\lambda _{N-1}$-almost every
point $y$ of $\partial \mathbb{H}$, either (i) $h$ has a finite
nontangential limit at $y$, or (ii) $h(S)=\mathbb{R}$ for every Stolz domain 
$S$ at $y$. Another well-known result (Theorem 1 of \cite{St61}) may be
formulated as follows.

\bigskip

\noindent \textbf{Theorem C (Stein) }\textit{Let }$h$ \textit{be a harmonic
function on }$\mathbb{H}$\textit{. Then, for} $\lambda _{N-1}$\textit{%
-almost every} \textit{point} $y$ \textit{of }$\partial \mathbb{H}$\textit{,
either }\newline
\textit{(i) }$h$\textit{\ has a finite nontangential limit at }$y$, \textit{%
or \newline
(ii) for every Stolz domain }$S$ \textit{at }$y$,%
\begin{equation*}
\int_{S}x_{N}^{2-N}\left\Vert \nabla h(x)\right\Vert ^{2}~d\lambda
_{N}(x)=\infty .
\end{equation*}

\bigskip

Let $\sigma $ denote surface area measure on level sets of (nonconstant)
harmonic functions. As before, the co-area formula shows that, for $h$ and $%
S $ as above,%
\begin{equation*}
\int_{S\cap \{a<h<b\}}x_{N}^{2-N}\left\Vert \nabla h(x)\right\Vert
^{2}~d\lambda _{N}(x)=\int_{(a,b)}\int_{S\cap \{h=t\}}x_{N}^{2-N}\left\Vert
\nabla h(x)\right\Vert ~d\sigma (x)d\lambda _{1}(t)
\end{equation*}%
whenever $a<b$. Hence the following analogue of Theorem \ref{main}
simultaneously strengthens the results of both Carleson and Stein.

\begin{theorem}
\label{hc}\textit{Let }$h$ \textit{be a harmonic function on }$\mathbb{H}$%
\textit{. Then, for} $\lambda _{N-1}$\textit{-almost every} point $y$ of $%
\partial \mathbb{H}$\textit{, either }\newline
\textit{(i) }$h$\textit{\ has a finite nontangential limit at }$y$, \textit{%
or \newline
(ii) for every Stolz domain }$S$ \textit{at }$y$,%
\begin{equation*}
\int_{S\cap h^{-1}(\{t\})}x_{N}^{2-N}\left\Vert \nabla h(x)\right\Vert
~d\sigma (x)=\infty
\end{equation*}%
for $\lambda _{1}$-almost every $t\in \mathbb{R}$.
\end{theorem}

\begin{proof}[Outline proof]
The proof of Theorem \ref{hc} is similar in pattern to that of Theorem \ref%
{main}, and simpler in some respects. We will therefore only give an outline
as a guide to the reader, and indicate a few points of difference. Let 
\begin{equation*}
S_{N}(y,\delta )=\{x\in \mathbb{H}:\sqrt{1-\delta ^{2}}\left\Vert
x-y\right\Vert <x_{N}<\delta \}\text{ \ \ \ }(y\in \partial \mathbb{H}%
,0<\delta <1).
\end{equation*}%
For each $j\geq 2$ we define the Borel function%
\begin{equation*}
M_{j}(y,t)=\int_{S_{N}(y,j^{-1})\cap h^{-1}(\{t\})}x_{N}^{2-N}\left\Vert
\nabla h(x)\right\Vert ~d\sigma (x)\text{ \ \ \ }(y\in \partial \mathbb{H}%
,t\in \mathbb{R}).
\end{equation*}%
The implicit function theorem allows us to express the set $h^{-1}(\{t\})$
locally on $\mathbb{H}\backslash \{\nabla h=0\}$ in the form%
\begin{equation*}
x_{p}=g(x_{1},...,x_{p-1},x_{p+1},...,x_{N},t)
\end{equation*}%
for some $p\in \{1,...,N\}$, where $g$ is continuously differentiable and $p$
is chosen so that $\partial h/\partial x_{p}\neq 0$. We can then argue as in
Lemma \ref{Borel} to see that $M_{j}$ is Borel measurable on $\partial 
\mathbb{H}\times \mathbb{R}$. Thus the sets%
\begin{equation*}
E_{j}=\left\{ y\in \partial \mathbb{H}:\int_{\mathbb{R}}\chi
_{\{M_{j}<\infty \}}(y,t)~d\lambda _{1}(t)=0\right\} \text{ \ \ \ }(j\geq 2)
\end{equation*}%
are Borel. The set $F$ of all boundary points at which $h$ has a finite
nontangential limit is also Borel. We suppose that there exists $j$ such
that $\lambda _{N-1}(B_{1})>0$, where $B_{1}=\partial \mathbb{H}\backslash
(E_{j}\cup F)$, and seek a contradiction.

Let $\rho >0$ be large enough so that $\lambda _{N-1}(B_{2})>0$, where%
\begin{equation*}
B_{2}=\left\{ y\in B_{1}:\int_{\mathbb{R}}\chi _{\{M_{j}\leq \rho
\}}(y,t)~d\lambda _{1}(t)>0\right\} ,
\end{equation*}%
and define%
\begin{equation*}
A(t)=\left\{ y\in B_{2}:M_{j}(y,t)\leq \rho \right\} \text{ \ \ \ }(t\in 
\mathbb{R}).
\end{equation*}%
Then%
\begin{eqnarray*}
\int_{\mathbb{R}}\lambda _{N-1}(A(t))~d\lambda _{1}(t) &=&\int_{\mathbb{R}%
}\int_{B_{2}}\chi _{\{M_{j}\leq \rho \}}(y,t)~d\lambda _{N-1}(y)d\lambda
_{1}(t) \\
&=&\int_{B_{2}}\int_{\mathbb{R}}\chi _{\{M_{j}\leq \rho \}}(y,t)~d\lambda
_{1}(t)d\lambda _{N-1}(y)>0,
\end{eqnarray*}%
so $\lambda _{N-1}(A(t_{0}))>0$ for some $t_{0}$. We choose a bounded subset 
$A_{1}(t_{0})$ of $A(t_{0})$ such that $\lambda _{N-1}(A_{1}(t_{0}))>0$. For
each $y\in A_{1}(t_{0})$ we know that%
\begin{equation}
\int_{S_{N}(y,j^{-1})\cap h^{-1}(\{t_{0}\})}x_{N}^{2-N}\left\Vert \nabla
h(x)\right\Vert ~d\sigma (x)\leq \rho .  \label{h1}
\end{equation}

If we define 
\begin{equation*}
I_{N}(x,\delta )=\{y\in \partial \mathbb{H}:x\in S_{N}(y,\delta )\}\ \ \
(0<\delta <1,0<x_{N}<\delta ),
\end{equation*}%
then $\lambda _{N-1}\left( I_{N}(x,j^{-1})\right) $ is proportional to $%
x_{N}^{N-1}$ and, as before, there is a positive constant $C(N,j)$ such that%
\begin{equation*}
\underset{%
\begin{array}{c}
x_{N}\rightarrow 0+ \\ 
x\in S_{N}(y,j^{-1})%
\end{array}%
}{\lim \inf }\frac{\lambda _{N-1}\left( I_{N}(x,j^{-1})\cap
A_{1}(t_{0})\right) }{x_{N}^{N-1}}\geq C(N,j)
\end{equation*}%
for $\lambda _{N-1}$-almost every point $y$ of $A_{1}(t_{0})$. We can choose 
$\varepsilon \in (0,j^{-1})$ and $A_{0}\subset A_{1}(t_{0})$ such that $%
\lambda _{N-1}(A_{0})>0$ and%
\begin{equation}
\lambda _{N-1}\left( I_{N}(x,j^{-1})\cap A_{1}(t_{0})\right) \geq \frac{%
C(N,j)}{2}x_{N}^{N-1}\text{ \ \ \ }(0<x_{N}<\varepsilon ,x\in \Omega ),
\label{h2}
\end{equation}%
where%
\begin{equation*}
\Omega =\bigcup\limits_{y\in A_{0}}S_{N}(y,j^{-1}).
\end{equation*}

We obtain a measure on $\mathbb{H}$ by defining $\mu =\Delta (h-t_{0})^{+}$
in the sense of distributions, and then use Green's identity to see that%
\begin{equation*}
\mu (J)=\int_{J\cap h^{-1}(\{t_{0}\})}\left\Vert \nabla h(x)\right\Vert
~d\sigma (x)\text{ \ \ for any open set }J\subset \mathbb{H}.
\end{equation*}%
(When $N\geq 3$ the set where $\nabla h=0$ need no longer be discrete, but
it is contained in an $(N-2)$-dimensional manifold: see pp.\thinspace
716,\thinspace 717 of \cite{Ku}.) From (\ref{h1}) and (\ref{h2}) we see that%
\begin{eqnarray*}
\rho \lambda _{N-1}(A_{1}(t_{0})) &\geq
&\int_{A_{1}(t_{0})}\int_{S_{N}(y,j^{-1})}x_{N}^{2-N}~d\mu (x)d\lambda
_{N-1}(y) \\
&\geq &\int_{\Omega }x_{N}^{2-N}\int_{A_{1}(t_{0})}\chi
_{S_{N}(y,j^{-1})}(x)~d\lambda _{N-1}(y)d\mu (x) \\
&\geq &\frac{C(N,j)}{2}\int_{\Omega }x_{N}~d\mu (x),
\end{eqnarray*}%
and so we can form the Green potential $u$ in $\mathbb{H}$ of the measure $%
\mu |_{\Omega }$ (see Theorem 4.2.5(iii) in \cite{AG}).

As before (by Corollary 4.3.3 and Theorems 4.3.8(i) and 4.3.5 of \cite{AG}),
the function%
\begin{equation*}
(h-t_{0})^{+}+\frac{u}{\sigma _{N}\max \{1,N-2\}},
\end{equation*}%
where $\sigma _{N}$ denotes the surface area of the unit sphere in $\mathbb{R%
}^{N}$, is positive and harmonic on $\Omega $. Although the conformal
mapping argument that we used previously is no longer available, it remains
true that any positive harmonic function on $\Omega $ has a finite
nontangential limit at $\lambda _{N-1}$-almost every point of $A_{0}$, by a
result of Brelot and Doob (Th\'{e}or\`{e}me 10 of \cite{BD}). Hence we can
proceed, as in the proof of Theorem \ref{main}, to obtain a contradiction.
\end{proof}

\bigskip 

\noindent \textit{Stephen J. Gardiner}

\noindent School of Mathematics and Statistics, 

\noindent University College Dublin, 

\noindent Belfield, Dublin 4, Ireland.

\noindent \textit{e-mail:} stephen.gardiner@ucd.ie

\bigskip 

\noindent \textit{Myrto Manolaki}

\noindent School of Mathematics and Statistics, 

\noindent University College Dublin, 

\noindent Belfield, Dublin 4, Ireland.

\noindent \textit{e-mail: }arhimidis8@yahoo.gr


\bigskip

\begin{thebibliography}{99}
\bibitem{AG} D. H. Armitage and S. J. Gardiner, \textit{Classical Potential
Theory.} Springer, London, 2001.

\bibitem{BD} M. Brelot and J. L. Doob, Limites angulaires et limites fines.
Ann. Inst. Fourier (Grenoble) 13 (1963), fasc. 2, 395--415.

\bibitem{Car} L. Carleson, On the existence of boundary values for harmonic
functions in several variables. Ark. Mat. 4 (1962), 393--399.

\bibitem{CL} E. F. Collingwood and A. J. Lohwater, \textit{The theory of
cluster sets.} Cambridge University Press, Cambridge 1966.

\bibitem{EG} L. C. Evans and R. F. Gariepy, \textit{Measure theory and fine
properties of functions.} CRC Press, Boca Raton, FL, 1992.

\bibitem{Gar} S. J. Gardiner, Universal Taylor series, conformal mappings
and boundary behaviour. Ann. Inst. Fourier (Grenoble) 64 (2014), 327--339.

\bibitem{GaMa0} S. J. Gardiner and M. Manolaki, A convergence theorem for
harmonic measures with applications to Taylor series. Proc. Amer. Math. Soc.
144 (2016), 1109--1117.

\bibitem{GaMa} S. J. Gardiner and M. Manolaki, Boundary behaviour of
Dirichlet series with applications to universal series. Bull. Lond. Math.
Soc. 48 (2016), 735--744.

\bibitem{GM} J. B. Garnett and D. E. Marshall, \textit{Harmonic measure}.
Cambridge University Press, Cambridge, 2005.

\bibitem{H67} W. K. Hayman, \textit{Research problems in function theory}.
Athlone Press, London 1967.

\bibitem{HL} W. K. Hayman and E. F. Lingham, \textit{Research problems in
function theory.} Springer, Cham, 2019.

\bibitem{Ku} \"{U}. Kuran, $n$-dimensional extensions of theorems on
conjugate functions. Proc. London Math. Soc. (3) 15 (1965), 713--730.

\bibitem{Maz} V. G. Maz'ja, \textit{Sobolev spaces}. Springer, Berlin, 1985.

\bibitem{Nes} V. Nestoridis, Universal Taylor series, Ann. Inst. Fourier
(Grenoble) 46 (1996), 1293--1306.

\bibitem{Ple} A. Plessner, \"{U}ber das Verhalten analytischer Funktionen am
Rande ihres Definitionsbereiches. J. Reine Angew. Math. 158 (1927), 219--227.

\bibitem{Pomm} C. Pommerenke, \textit{Boundary behaviour of conformal maps}.
Springer, Berlin, 1992.

\bibitem{Spe} D. C. Spencer, A function-theoretic identity. Amer. J. Math.
65 (1943), 147--160.

\bibitem{St61} E. M. Stein, On the theory of harmonic functions of several
variables. II. Behavior near the boundary. Acta Math. 106 (1961), 137--174.

\bibitem{St82} E. M. Stein, The development of square functions in the work
of A. Zygmund. Bull. Amer. Math. Soc. (N.S.) 7 (1982), 359--376.
\end{thebibliography}
\end{document}